\documentclass[11pt]{amsart}
\usepackage{a4wide,amsmath,amsthm,amsfonts,amssymb}
\usepackage{hyperref}
\newtheorem{theorem}{Theorem}[section]
\newtheorem{lemma}[theorem]{Lemma}
\newtheorem{corollary}[theorem]{Corollary}
\newtheorem{proposition}[theorem]{Proposition}

\theoremstyle{definition}

\newtheorem{remark}[theorem]{Remark}
\newtheorem{example}[theorem]{Example}

\def\phi{\varphi}

\def\Z{\mathbb Z}
\def\Q{\mathbb Q}
\def\C{\mathbb C}
\def\T{\mathbb T}

\def\ol{\overline}

\def\int{\operatorname{int}}
\newcommand{\w}{\omega}

\begin{document}

\author[T.~Banakh]{Taras Banakh}
\address{T. Banakh: Ivan Franko National University of Lviv (Ukraine), and Institute of Mathematics, Jan Kochanowski University in Kielce (Poland)}
\email{t.o.banakh@gmail.com}

\author[S.~Bardyla]{Serhii~Bardyla}
\footnote{The work of the second author is supported by the Austrian Science Fund FWF (Grant  I
3709 N35).}
\address{S. Bardyla: Institute of Mathematics, Kurt G\"{o}del Research Center, Vienna, Austria}
\email{sbardyla@yahoo.com}

\author[I.~Guran]{Igor Guran}
\author[O.~Gutik]{Oleg Gutik}
\address{I. Guran, O. Gutik: Faculty of Mathematics, National University of Lviv,
Universytetska 1, Lviv, 79000, Ukraine}
\email{igor-guran@ukr.net, o\underline{\hskip5pt}\,gutik@franko.lviv.ua,
ovgutik@yahoo.com}

\author[A.~Ravsky]{Alex Ravsky}
\address{A. Ravsky: Department of Analysis, Geometry and Topology, Pidstryhach Institute for Applied Problems of Mechanics and Mathematics
National Academy of Sciences of Ukraine,
Naukova 3-b, Lviv, 79060, Ukraine}
\email{alexander.ravsky@uni-wuerzburg.de}

\title[Positive answers to Koch's problem in special cases]
{Positive answers to Koch's problem in special cases}


\begin{abstract}
A topological semigroup is monothetic provided it contains a dense cyclic subsemigroup.
The Koch problem asks whether every locally compact monothetic monoid
is compact. This problem was opened for more than sixty years, till in 2018 Zelenyuk
obtained a negative answer. In this paper we obtain a positive answer for Koch's problem
for some special classes of topological monoids. Namely, we show
that a locally compact monothetic topological monoid is
a compact topological group if and only if $S$ is a submonoid of a quasitopological group if and only if $S$ has open shifts if and only if $S$ is non-viscous in the sense of Averbukh. The last condition means that any neighborhood $U$ of the identity $1$ of $S$ and for any element $a\in S$ there exists a neighborhood $V$ of $a$ such that any element $x\in S$ with $(xV\cup Vx)\cap V\ne\emptyset$ belongs to the neighborhood $U$ of 1.
\end{abstract}
 \maketitle
\section{Introduction}

In this paper by a ``space" means a ``topological space".
Any space in the paper is assumed to be Hausdorff, unless
it satisfies others explicitly stated separation axioms.
For the semigroup operation we use the additive notation when it is necessarily commutative
and the multiplicative notation, otherwise. By a {\em topologized semigroup} we understand a semigroup endowed with a topology.

\subsection{Koch's problem.} A topologized semigroup $S$ is called \emph{monothetic} (cf.,~\cite[p.10]{CHK})
provided it contains a dense cyclic subsemigroup, that is
$S=\overline{\{a^n:n\ge 1\}}$
for some element $a\in S$, which is called a \emph{generator} of the semigroup $S$.
Compact monothetic (semi)topological groups are compact abelian
groups whose character groups are subgroups of the circle (endowed with the discrete topology).
Pontryagin's theorem~\cite{Pon2} states that each locally compact monothetic topological group is
compact. Compact monothetic topological semigroups were described by Hewitt
~\cite{Hew}, see also ~\cite[p.123-130]{CHK}.
In particular, a  locally compact monothetic topological semigroup with a minimal ideal
is compact, see~\cite[Theorem 3.13]{CHK}.
Whether a counterpart of Pontryagin's theorem holds for locally compact monothetic topological monoids
was a well-known problem, posed by Koch~\cite{Koc}.
Remark that if such a monoid is compact, then it is a (topological) group~\cite{Koc},
 \cite{Ell}.
In \cite{Zel} Zelenyuk constructed a countable locally compact cancellative
monothetic topological semigroup which is neither compact, nor discrete. It is easy to see that
a countable locally compact monothetic topological monoid is discrete,
so we cannot attach the identity to Zelenyuk's semigroup.
It took him thirty years to reach a final (negative) answer to Koch's problem in \cite{Zel2}.

On the other hand, Koch's problem has an affirmative answer in some special cases. For instance,
for a totally disconnected monoid $S$, which was known already to Koch~\cite{Koc}.
If $S$ is a group, then $S$ is a topological group by Ellis' Theorem~\cite{Ell2}.
Recently Guran and Kisil \cite{GK} improved this result showing that
a locally compact cancellative monothetic topological monoid with open shifts is
a compact topological group.

In the present paper we obtain positive answers to Koch's Problem
for more general classes of topological monoids. In particular,
we show that a monothetic  topological monoid is cancellative
provided it has open shifts or is non-viscous.

\subsection{Embeddings of semigroups into groups.} Each submonoid of a group is
a cancellative semigroup. On the other hand, it is well-known (see, for instance,
\cite{CP}) that a commutative cancellative semigroup can be embedded
into the group of its quotients.

Unfortunately, this algebraic construction can has no obvious topological counterparts.

The problem of recognizing topological semigroups that embed into topological
groups is a well-known and old research problem. A list of papers published before
1990 is contained in~\cite{Law}. Positive results are known for some special cases, see \cite{McKil}, \cite{Rot}, \cite{LLZ}.

A recent advance~\cite{Ave1}--\cite{Ave4} by Averbukh allowed to find a surprisingly
simple characterization of monothetic topological monoids which can be embedded
into a topological group. Namely, such monoids are cancellative, $T_3$, and non-viscous.
The last property is a simple condition expressed in terms of elements and neighborhoods.
Moreover, this notion turned out to be useful  in obtaining a positive answer to the Koch problem for
a partial case. Namely, a monothetic non-viscous topological monoid $S$ is cancellative
(see Proposition~\ref{prop:non-viscous-cancellativity}) and if $S$ is locally compact,  then
it is a compact topological group, see Theorem~\ref{thm:our-Koch}.

\section{Definitions}

A semigroup $S$ is \emph{cancellative} provided for any elements $a,x,y\in S$, the equality $ax=ay$ or $xa=ya$ implies $x=y$.

A topologized semigroup $S$ is called
\begin{itemize}
\item a \emph{left-topological semigroup} provided for every $a\in S$
the left shift $S\to S$, $x\mapsto ax$, is continuous;
\item a \emph{semitopological semigroup} provided the multiplication
$S\times S\to S$, $(x,y)\mapsto xy$, is separately continuous;
\item a \emph{topological semigroup} if the multiplication $S\times S\to S$ is continuous.
\end{itemize}
A topology on a semigroup $S$ is called a \emph{semigroup topology} on $S$ if it turns $S$ into a topological semigroup.

A topological group $G$ is called
\begin{itemize}
\item a \emph{semitopological group} provided the multiplication
$S\times S\to S$ is separately continuous;
\item a \emph{quasitopological group} if $G$ is a semitopological group with continuous inversion $G\to G$, $x\mapsto x^{-1}$;
\item a {\em paratopological group} if the multiplication $S\times S\to S$ is continuous;
\item a \emph{topological group} if $G$ is a paratopological group and a quasitopological group.
\end{itemize}

A standard example of a paratopological group failing to be a topological group is
the Sorgenfrey line, that is the real line endowed with the Sorgenfrey topology
(generated by the base consisting of half-intervals $[a,b)$, $a<b$).

Since compact spaces have very good properties, they play an important role in general topology.
These properties are so good that even spaces satisfying slightly weaker conditions
are still useful. Therefore, topologists often investigate various classes of compact-like spaces
and relations between them, see, for instance, basic \cite[Chap. 3]{Eng} and
general works \cite{vanDouwenReedRoscoeTree1991}, \cite{Matveev1998},
\cite{VaughanHSTT}, \cite{StephensonJr1984}, \cite{Lipparini2016},\cite{GR2}.
In the present paper we will consider some of compact-like properties.

We recall that a space $X$ is called
\begin{itemize}
\item {\it sequentially compact} if each sequence of $X$ contains
a convergent subsequence;
\item {\it countably compact at a subset} $A$ of $X$ if each infinite
subset $B$ of $A$ has an accumulation point $x$ in the space $X$
(the latter means that each neighborhood of $x$ contains infinitely many points of the set $B$);
\item {\it countably compact} if $X$ is countably compact at itself;
\item {\it countably pracompact} if $X$ is countably compact at a dense subset of $X$;
\item {\it feebly compact} if each locally finite family of open subsets of the space $X$ is finite.
\end{itemize}

Relations between various classes of compact-like spaces are well-studied. Some of them
are presented at Diagram 3 in ~\cite[p.17]{Matveev1998}, at Diagram 1 in
~\cite[p. 58]{Dorantes-AldamaShakhmatov} (for Tychonoff spaces), at Diagram 3.6 in~\cite[p.
611]{StephensonJr1984}, and at Diagram in\cite{GR2}.

In particular, the following inclusions hold.
\begin{itemize}
\item Each compact space is countably compact.
\item Each sequentially compact space is countably compact.
\item Each countably compact space is countably pracompact.
\item Each countably pracompact space is feebly compact.
\end{itemize}

In these terms, a space $X$ is compact if and only if $X$ is countably compact and Lindel\"of.
A Tychonoff space $X$ is feebly compact iff it is pseudocompact, that is
iff each continuous real-valued function on $X$ is bounded.


\section{The non-viscousity}

Following Averbukh \cite{Ave1} we define a topologized monoid $S$ to be {\em non-viscous} if for any neighborhood $U$ of the identity $1$ of $S$ and for any element $a\in S$
there exists a neighborhood $V$ of $a$ such that any element $x\in S$ with $(xV\cup Vx)\cap V\ne\emptyset$ belongs to $U$.

\begin{remark}\label{rem:non-viscous-conditions}
It is easy to see that any non-viscous $T_1$ topologized monoid $S$ satisfies the following Condition:
for any $a,x\in S$ if $ax = a$ or $xa = a$ then $x = 1$.
In particular, the only idempotent of $S$ is the identity.
On the other hand, a commutative monoid generated by two elements $p$ and $q$ with the
only condition $p+q=p$ has only an identity idempotent but does not satisfy Condition.
By ~\cite[Proposition 1.1]{Ave4}, each compact topological monoid satisfying Condition
is non-viscous. Condition also assures that $S$ is non-viscous provided $S$ is discrete.
On the other hand, Condition holds for every cancellative monoid,
so the additive monoid of non-negative real numbers endowed with the usual
topology, but with isolated zero is a viscous locally compact cancellative topological monoid.
On the other hand, the
direct product of additive semigroup $\mathbb{N}$ by the semilattice $(\{0,1\},\min)$
with adjoint identity satisfies Condition but is not cancellative.
Nevertheless, it is easy to check that in a finite monoid whose only idempotent is identity,
each element is invertible, so this monoid is a group, see, for instance, ~\cite[$\S$ 3]{Kur}.
\end{remark}

\begin{lemma}\label{lem:auto-cancellativity} Let $S$ be a monothetic semitopological monoid.
Let $S=\overline{P}$, where $P$ is the semigroup generated by an element $p\in S$.
If  $\overline{x+P}\cap P\ne\varnothing$ for each $x\in S$, then $S$ is cancellative.
\end{lemma}
\begin{proof}
Suppose to the contrary that there exist elements $a$, $b$, and $x$ of the
semigroup $S$ such that $a\ne b$, but $a+x=b+x$. Since the semigroup $S$ is Hausdorff,
the set $Y=\{y\in S:a+y=b+y\}$ is closed. Since $Y\supset x+P$, there exists a natural
number $n$ such that $np\in Y$. Since $Y$ is an ideal, $mp\in Y$ for each $m\ge n$.
Since $a\ne b$ and $S$ is a Hausdorff semitopological monoid, there exists a neighborhood
$U$ of the identity such that sets $a+U$ and $b+U$ are disjoint.
Since the semigroup $S$ is monothetic, there exist natural $m\ge n$ such that $mp\in U$.
Since $mp\in Y$, $a+mp=b+mp\in (a+U)\cap (b+U)$, a contradiction.
\end{proof}

\begin{corollary} A monothetic semitopological monoid $S$ is cancellative,  if for each $x\in S$ the set
$\overline{x+S}$ has non-empty interior in $S$.
\end{corollary}

\begin{proof} It suffices to remark that since $P$ is dense in $S$, $x+P$ is dense in
$x+S$ for each $x\in S$, so $\overline{x+S}=\overline{x+P}$.
\end{proof}

\begin{corollary}\label{cor:open-shifts-cancellative}
A monothetic semitopological monoid with open shifts is cancellative.
\end{corollary}

\begin{proposition}\label{prop:non-viscous-cancellativity} A monothetic non-viscous topological monoid
$S$ is cancellative.
\end{proposition}
\begin{proof}
Suppose to the contrary that there exist elements $a$, $b$, and $x$ of the
semigroup $S$ such that $a\ne b$, but $a+x=b+x=y$. Since $S$ is a Hausdorff topological semigroup,
there exists neighborhoods $V_a$ of $a$, $V_b$ of $b$, and $U$ of $1$ such that
the sets $V_a+U$ and $V_b+U$ are disjoint.
Since $S$ is not viscous, there exists a neighborhood $V$ of the point $y$ such that for any  $y_1, y_2\in V$ and
$z\in S$ such that $y_1+z=y_2$ we have $z\in U$.
Since $S$ is a topological semigroup,
there exist neighborhoods $W_a$ of $a$, $W_b$ of $b$, and $V_x$ of $x$ such that
$W_a+V_x\subset V$ and $W_b+V_x\subset V$.
Let $S=\overline{P}$, where $P$ is the semigroup generated by an element $p\in S$.
There exist natural numbers $n_a$, $n_b$ and $n_x$ such that $n_ap\in V_a\cap W_a$,
$n_bp\in V_b\cap W_b$, and $n_xp\in V_x$. Then $n_ap+n_xp\in V$ and $n_bp+n_xp\in V$.
Without loss of generality we may assume that $n_b>n_a$. Then
$(n_bp+n_xp)-(n_ap+n_xp)=(n_b-n_a)p\in U$. So $V_b+U\supset V_b\ni n_bp=n_ap+(n_b-n_a)p\in V_a+U$,
a contradiction.
\end{proof}

\begin{lemma}\label{lem:mon-dis} A $T_1$ monothetic left topological monoid $S$ containing
an isolated point $a$ is a finite group.
\end{lemma}
\begin{proof} Let $S=\overline{P}$, where $P$ is the smallest semigroup generated by an element
$p\in S$. Then $a=p^n$ for some natural $n$ and $p^nU=\{p^n\}$ for some neighborhood $U$ of
the identity. Pick any $p^m\in U$  with $m>0$. Then $p^{m+n}=p^n$,
which implies that the semigroup $P$ is finite. Since $S$ is a $T_1$ space with a finite dense
set $P$, we see that $S=P$ is a finite monoid. Since $S$ is monothetic, $S$ is a group (see,
for instance,~\cite[$\S$ 3]{Kur}).
\end{proof}

\begin{lemma}\label{lem:nvis-open-shifts} A locally compact non-viscous monothetic topological monoid $S$ has open shifts.
\end{lemma}

\begin{proof} By Lemma~\ref{lem:mon-dis} it suffices to consider the case when $S$ has no isolated points.

Let $U$ be an arbitrary compact neighborhood of the zero and $a\in S$ be an arbitrary element. Since $S$ is non-viscous, there exists a neighborhood $V$ of the point $a$ such that for any  $a_1, a_2\in V$ and $x\in S$ with $a_1+x=a_2$ we have $x\in U$. Since the semigroup $S$ is monothetic, the set $V\cap P$ is dense in $V$. Let $np\in V\cap P$ be an arbitrary element. Since $S$ has no isolated points,
the set $V'=V\cap P\setminus \{mp:0\le m<n\}$ is dense in $V$, too. Then
$V\subset \overline{V'}=\overline{(V'-np)+np}\subset \overline{U+np}=U+np$.

We claim that $V\subset a+U$. To derive a contradiction, assume that there exists an element $v\in V\setminus(a+U)$. By the compactness of $U$, there exists an open neighborhood $W$ of $a$  such that $v\not\in W+U$. Pick an arbitrary point $n'p\in W\cap V'$.
By the previous paragraph, $v\in V\subset n'p+U\subset W+U$, a contradiction. Thus $a\in V\in a+U$ and $a+U$ is a neighborhood of the point $a$.

Now let $y\in S$ be an arbitrary element and $V\subset S$ be an arbitrary open set. Let $z\in y+V$ be an arbitrary point. Then $z=y+v$ for some point $v\in V$. Pick a compact neighborhood $U$ of the zero such that $v+U\subset V$. Then $z=y+v\in y+v+U=z+U\subset y+V$, and by the previous paragraph the set $z+U$ is a neighborhood of the point $z$.
\end{proof}

Now we can give a partial answer to the Koch problem.

\begin{theorem}\label{thm:our-Koch} For a locally compact monothetic topological monoid $S$, the following conditions are equivalent:
\begin{enumerate}
\item $S$ is a compact topological group;
\item $S$ has open shifts;
\item $S$ is non-viscous.
\end{enumerate}
\end{theorem}

\begin{proof} The implication $(1)\Rightarrow(3)$ is trivial and  $(3)\Rightarrow(2)$ follows from Lemma~\ref{lem:nvis-open-shifts}. To prove that $(2)\Rightarrow(1)$, assume that $S$ has open shifts. By Corollary~\ref{cor:open-shifts-cancellative}, the monothetic monoid $S$ is cancellative. By Theorem 8 of \cite{GK}, $S$ is either discrete or a compact topological group. If $S$ is discrete, then $S$ is a finite group by  Lemma~\ref{lem:mon-dis}.
\end{proof}

\begin{example}
Averbukh proved~\cite[Theorem 2.2]{Ave4} that
a monothetic topological monoid $S$ can be embedded into a topological group if and only if $S$ is
$T_3$, cancellative, and non-viscous. Proposition~\ref{prop:non-viscous-cancellativity}
implies that the cancellativity holds automatically, but we show that each of the remaining
conditions is essential by providing a counterexample for the case when it is dropped.
\begin{itemize}
\item (\emph{Monotheticity}) Let $S=[0,\infty)$ be the (non-viscous) submonoid
of the Sorgenfrey arrow, see Example~\ref{exam:Sorgenfrey}.
To derive a contradiction, assume that $S$ is a submonoid of a topological group $G$.
Then there exists a neighborhood $U$ of the identity $0$ of $G$ such that
$(1\cdot U)\cap S\subset [1,2)$.
On the other hand, since $G$ is a topological group, there exists a number
$0<\varepsilon<1$ such that $U\supset [0,\varepsilon)^{-1}$.
Then $(1\cdot U)\cap S\supset (1-\varepsilon, 1)$, a contradiction.
\item (\emph{$T_3$}) Let $S=\{0\} \cup[1,\infty)$ be the submonoid of the group of the real numbers.
Define a base of a (semigroup) topology on $S$ by taking at each element $x\in S\setminus\{1\}$
a base $\{(x-\varepsilon,x+\varepsilon)\cap S:\varepsilon>0\}$
of a subspace $S$ of $\mathbb R$ endowed with the standard topology, and at the
element $1$ a base $\{[x,x+\varepsilon)\setminus\{1+1/n:n\in\mathbb N\}:\varepsilon>0\}$.
\item (\emph{Non-viscousity}) Let $S$ be Sorgenfrey circle, see Example~\ref{exam:Sorgenfrey}.
\end{itemize}
\end{example}

We shall call an element $g$ of a semitopological monoid $S$ \emph{topologically periodic},
if for each neighborhood $U$ of the identity of $S$ there exists a natural number $n$
such that $g^n\in U$. A semitopological monoid $S$ is called \emph{topologically periodic},
provided each its element is topologically periodic. It is easy to see that each
(pre)compact topological group $G$ is topologically periodic.

\begin{proposition}\label{prop:non-viscous-cc-stm-is-top-per}
Each non-viscous countably compact topologized monoid $S$ is topologically periodic.
\end{proposition}
\begin{proof} Let $g$ be any element of $S$ and $U$ be any neighborhood of the identity $1$ of $S$.
Let $a$ be a cluster point of the sequence $\{g^n:n\in\mathbb N\}$. Let $V$ be a
neighborhood of $a$ such that for any elements $a_1,a_2\in V$, if
$a_1x = a_2$ then $x\in U$. There exist numbers $m<n$ such that $g^m,g^n\in V$.
Then $g^{m-n}\in U$.
\end{proof}

\begin{proposition}\label{prop:tm-square-cc-is-top-per}
Let $S$ be a topological monoid such that $S\times S$ is countably compact.
Then each left invertible element $g$ of $S$ is topologically periodic.
\end{proposition}
\begin{proof} Pick an element $h\in S$ such that $hg=1$. Let $(a,b)$ be a cluster point
of the sequence ${(h^n,g^n)}$.
Assume that $ab\ne 1$. Then $S\setminus \{1\}$ is a neighborhood of
$ab$. There exists neighborhoods $U$ and $V$ of $a$ and $b$, respectively such that
$UV\subset S\setminus\{1\}$. Since $U\times V$ is a neighborhood of $(a,b)$ is $S\times S$,
there exists natural $n$ such that $(h^n,g^n)\in U\times V$. Then $1=h^ng^n\in S\setminus \{1\}$,
a contradiction.
Thus $ab=1$. Let $W$ be any neighborhood of $ab$. There exists neighborhoods $U$ and $V$ of $a$ and $b$,
respectively such that $UV\subset W$. Since $U\times V$ is a neighborhood of $(a,b)$ is $S\times S$,
there exists natural $m<n$ such that $(h^m,g^n)\in U\times V$.
Then $g^{n-m}=h^mg^n\in W$.
\end{proof}

\begin{proposition}\label{prop:top-per-stm-is-cancel}
Each topologically periodic semitopological monoid $S$ is cancellative.
\end{proposition}
\begin{proof} Let $a$ and $b$ be distinct elements of the semigroup $S$ and $x\in S$.
Since $a\ne b$, there exists a neighborhood $U$ of the identity of $S$ such that
$aU\cap bU=\varnothing=Ua\cap Ub$. Since element $x$ is topologically
periodic, there exists a positive integer $n$ such that $x^n\in U$. Then $ax^n\in aU$, and $bx^n\in
bU$, so $ax^n\ne bx^n$. Thus $ax\ne bx$. Similarly we can show that $xa\ne xb$.
\end{proof}

The following example shows that local compactness is essential in Theorem~\ref{thm:our-Koch},
It also concerns~\cite[Theorem 2.2]{Ave4} and shows that a monoid with open
shifts, contained in a topological group, may fail to be a group.

\begin{example}\label{exam:open-shifts} Let $g$ be any non-perodic topologically periodic
element of a topological group $G$. For instance, $g$ can be any non-periodic element
of any (pre)compact topological group $G$. In particular, we
can take for $G$ the unit circle $\T=\{z\in\C:|z|=1\}$ and for $g$ any element of $\T$ such that
$\arg g/\pi$ is irrational. Put $S=\{g^n: n\in\omega\}$. Since $S$ is a subsemigroup
of a topological group, $S$ is a cancellative topological monoid. Moreover,
it is easy to check that each topological group is non-viscous and
a submonoid of a non-viscous monoid is non-viscous.
Since $g$ is topologically periodic, $S$ is monothetic.
Since $g$ is non-periodic, $S$ is not a group. For any element $h\in S$ the set $hS$
has finite complement in $S$ and hence is open in $S$. Let $U$ be any open subset of $S$. Since
a shift in a topological group is a homeomorphism onto its image,
a set $hU$ is open in $hS$, so it is open is $S$ too.
\end{example}

On the other hand, a non-viscous monothetic submonoid of a topological group
needs not have open shifts.

\begin{example}\label{exam:non-open-shifts} Let $s:\T\to\Q$ be a surjective homomorphism,
$S'=\{z\in\T:s(z)>0\}$ and $S=S'\cup\{0\}$. Then $\T=S'-S'$, both sets $S'$ and $-S'$ are dense in $\T$, but
$S'\cap (-S')=\varnothing$. Thus $S$ is a non-viscous monoid which is not open in $\T$.
Since the set $S'$ is uncountable, it contains a non-periodic element, which generates a
cyclic semigroup which is dense in $\T$, so in $S$ too. Hence the monoid $S$
is monothetic. Pick any element $g\in S'$. Then the set $S'\setminus (g+S')=s^{-1}((0,s(g)])$
is dense in $\T$ and in $S$. Thus the set $g+S$ is not open in $S$.
\end{example}

We can generalize Proposition 1.2 from~\cite{Ave4}, which states that a non-viscous paratopological group
is a topological group as follows.

\begin{proposition}\label{prop:non-viscous-stg-is-tg}
A non-viscous left-topological group $G$ is a topological group.
\end{proposition}

\begin{proof} If suffices to show that the map $\delta:G\times G\to G$, $\delta:(x,y)\mapsto xy^{-1}$, is continuous.
Fix any points $x,y\in G$ and an open neighborhood $W\subset G$ of $xy^{-1}$. Since $G$ is left-topological, $yx^{-1}W$ is an open neighborhood of $yx^{-1}xy^{-1}=1=yy^{-1}$. By the non-viscousity of $G$, the point $y$ has an open neighborhood $V_y$ such that $V_yV_y^{-1}\subset yx^{-1}W$ and hence $(xy^{-1}V_y)V_y^{-1}\subset W$. Since $G$ is left-topological, the set $V_x:=xy^{-1}V_y$ is an open neighborhood of $x$. Since $V_xV_y^{-1}\subset W$, the map $\delta$ is continuous at $(x,y)$.
\end{proof}

\begin{proposition} Let $S$  be a non-viscous topological monoid $S$ such that
$S$ is sequentially compact or $S^{2^\mathfrak c}$ is countably compact.
Then $S$ is a topological group.
\end{proposition}
\begin{proof}
By Proposition~\ref{prop:non-viscous-cc-stm-is-top-per},
$S$ is topologically periodic. By Proposition~\ref{prop:top-per-stm-is-cancel},
$S$ is cancellative.
If $S$ is sequentially compact, then $S$ is a topological group by~\cite[Theorem 6]{BG}.
If $S^{2^\mathfrak c}$ is countably compact, then by Theorem 3.1 and the remark on the beginning of
Section 3.2 of~\cite{Tom2}, $S$ is a group. By~\ref{prop:non-viscous-stg-is-tg}, $S$ is a topological group.
\end{proof}

\begin{example}\label{exam:Sorgenfrey} Let $\mathbb S$ be  the Sorgenfrey line, that is the real
line endowed with the Sorgenfrey topology (generated by the base consisting of half-intervals
$[a,b)$, $a<b$). It is easy to check that a submonoid $S=[0,\infty)$ of $\mathbb S$ is
non-viscous.

On the other hand, let $\mathbb T=\{|z|\in\mathbb C:|z|=1\}$ be the unit circle whose group operation
is the multiplication of complex numbers. Define a map $f:\mathbb S\to\mathbb T$,  $f:x\mapsto e^{2\pi xi}$. It is well-known that $f$ is a homomorphism with the kernel $\mathbb Z$.
Therefore the group $G=\mathbb S/\Z$ endowed with the quotient topology is a paratopological
group (see, for instance,~\cite[Proposition 1.12]{Rav1}). We call it \emph{the Sorgenfrey circle}.
The group $G$ is monothetic. Since $G$ is not a topological group, it is viscous.

Now define an action $\mathbb Z_2\to \operatorname{Aut}(\mathbb T)$ of the two-element group
$\mathbb Z_2$ on $\mathbb T$ by putting $\sigma(1)(z)=z^{-1}$ for each $z\in\mathbb T$. The first
author discovered in~\cite{Ban}  that the semidirect product $H=\mathbb T\ltimes\mathbb Z_2$
naturally endowed with the topology of the Aleksandrow ``two-arrows" space (see, for instance,
~\cite[3.10.C]{Eng}) is a left-topological group. The dense subgroup $\mathbb T\times \{0\}$ of $H$
is can be identified with the Sorgenfrey circle. Any non-periodic element of $H$ generates a dense
subsemigroup in $H$, which implies that $H$ is monothetic, but not abelian. So, $H$ is a monothetic
first-countable compact left-topological group, which is not semitopological. By Proposition
\ref{prop:non-viscous-stg-is-tg}, the left-topological group $H$ is viscous.
\end{example}

\section{Embedded variations}\label{sec:embedded}

To obtain another partial answer to the Koch problem, we need the following

\begin{lemma}\label{lem:new-bardyla} For any topologically periodic element $g$ of a quasitopological group $G$ and $n\in\mathbb Z$, the set $H=\overline{\{g^m:m\ge n\}}$ is a group.
\end{lemma}

\begin{proof}
If $g$ is periodic then $H$ is a finite group, so we can assume that
$g$ is not periodic. Let $l$ be any integer and $U$ be any neighborhood of an element $g^l$.
Since $1\cdot g^l=g^l$, there exists a neighborhood $V$ of $1$ such that $Vg^l\subset U$.
Since $g$ is topologically periodic, there exists a number $k\ge m-l$ such that $g^k\in V$.
Then $g^kg^l\in \{g^n:n\ge m\}\cap U$. Thus $g^l\in H$. It follows $H$ is the closure of
the subgroup $\{g^l:l\in\Z\}$ in the quasitopological group $G$, so by~\cite[Prop. 1.4.13]{AT}
$H$ is a group.
\end{proof}

\begin{theorem}\label{thm:Bardyla}
Assume that a monothetic topological monoid $S$ is locally compact at the unit $1$. If $S$ is a submonoid of a quasitopological group $G$, then $S$ is a compact topological group.
\end{theorem}
\begin{proof}
Let $g$ be a generator of the monoid $S$ and $H=\overline{\{g^m:m\ge 1\}}^G$.
By Lemma~\ref{lem:new-bardyla}, $H$ is a (quasitopological) group. It is easy to see that
$\ol{S}^G=H$. Let $U$ be an open (in $S$) neighborhood of the identity
such that $\ol{U}^S$ is compact. Then $S\supset \ol{U}^S=\ol{U}^H$ is a neighborhood
of the identity in $H=\ol{S}^H$. Let $h\in H$ be any element. Since $H$ is a
quasitopological group, a set $hS^{-1}$ is a neighborhood in $H$ of the element $h$,
so $hS^{-1}\cap S\ne\varnothing$, thus $h\in SS=S$.
Therefore $H=S$ is both a quasitopological group and a topological semigroup,
that is, a topological group. Since $H$ is locally compact, it is compact by the Pontryagin
alternative.
\end{proof}

Remark that Theorem~\ref{thm:Bardyla} generalizes~\cite[Cor.1.15]{Ave1} to quasitopological
groups and that by~\cite{BGR2} every locally compact topological group $G$ is  closed in any quasitopological group which contains $G$ as a subgroup.

\begin{corollary}\label{cor:Bardyla} Assume that a topologically periodic topological
monoid $S$ is locally compact at the unit $1$.
If $S$ is a submonoid of a quasitopological group $G$, then $S$ is a topological group.
\end{corollary}
\begin{proof} Let $a$ be an arbitrary element of $S$. Then $S_a$, the smallest
closed subsemigroup of $S$ containing $a$ is a monothetic topological monoid,
whose identity has a neighborhood with compact closure. Since $S_a$ is contained
in a quasitopological group, by Theorem ~\ref{thm:Bardyla}, it is a group.
In particular, $S_a$ (and so, $S$) contains
the inverse of $a$, so $S$ is a group, see~\cite[$\S$ 3]{Kur}.
Thus $S$ is a paratopological group with continuous inversion, because $S$ is a
subgroup of a quasitopological group.
\end{proof}

A counterpart of Theorem~\ref{thm:Bardyla} fails for $T_1$ case, as shows the following
example.

\begin{example} An additive group $\Z$ endowed with the cofinite topology is a $T_1$
compact quasitopological group containing a monothetic compact topological semigroup $S$
of non-negative integers. The monoid $S$ cannot be embedded in a $T_1$ topological
group, because $S$ is not Hausdorff.
\end{example}

\begin{remark}\label{rem:Bardyla} Theorem~\ref{thm:Bardyla} suggests a question whether
we can weaken the local compactness at 1 to the (local) countable or feeble compactness.
Our answers to this question are the same as the answers to a known question when a compact-like cancellative topological semigroup $S$ is a group.

Bokalo and Guran \cite{BG} proved that each sequentially compact cancellative topological semigroup is a topological group.

For pseudocompact case there is a relatively simple counterexample.
Let $H$ be the Tychonoff product of an uncountable family of compact topological groups,
$G$ be the $\Sigma$-product contained in $H$. Assume that $H$ is monothetic and pick an element $g\in H$
such that the subgroup of $H$ generated by $g$ is dense in $H$. Clearly, $g\not\in G$.
Since $G$ is a dense countably compact subset
of $H$, the subsemigroup $S$ of $H$ generated by $G$ and $g$ is countably pracompact, so it
is pseudocompact. Since $S\ni g$, it is easy to see that the subsemigroup $S$ is monothetic.
By Exercise 9.6.f from~\cite{AT}, a compact abelian group $G$ is monothetic iff
the dual group $G^*$ is isomorphic to a subgroup of the group $\mathbb T$ endowed with the
discrete topology.
We can also directly show that a group $\mathbb T^\kappa$ endowed with the Tychonoff product topology
is monothetic for each $\kappa\le\mathfrak c$. It is well-known that the real line considered as the linear space over the field $\mathbb Q$ has a Hamel basis of cardinality continuum and hence contains a linearly independent subset $T=\{t_\alpha:\alpha<\kappa\}$ consisting of pairwise distinct real numbers $t_\alpha$, $\alpha<\kappa$.
Put $g=(g_\alpha)=(e^{2\pi t_\alpha i})\in \mathbb T^\kappa$. By~\cite[Example 65]{Pon},
for each sequence $t_{\alpha_1},\dots, t_{\alpha_r}$ of mutually different elements of $T$,
each sequence $d_1,\dots, d_r$ of real numbers, and each $\varepsilon>0$
there exist a sequence $n_1,\dots, n_r$ of integer numbers and an integer number $m$ such
that $|mt_{\alpha_i}-d_i-n_i|<\varepsilon$ for each index $i=1,\dots, r$.
This implies that the set of powers of $g$ is dense in the group $\mathbb T^\kappa$.

The countably compact case is much more subtle.

Denote by TT the following axiomatic assumption: there is an infinite torsion-free abelian
countably compact topological group without non-trivial convergent sequences. An example of such a group was constructed by M.~Tkachenko \cite{Tka}
under the Continuum Hypothesis. Later, the Continuum Hypothesis weakened  to the
Martin Axiom for $\sigma$-centered posets by Tomita in \cite{Tom2}, for countable posets in
\cite{KTW}, and finally to the existence continuum many incomparable selective
ultrafilters in \cite{MT}. Yet, the problem of the existence of a countably compact
group without convergent sequences in ZFC seems to be open,  see \cite{DS}.

The proof of \cite[Lemma 6.4]{BDG} implies that under TT there exists a
group topology on a free abelian group $F$ generated by the cardinal $\mathfrak c$
such that for each countable infinite subset $M$ of the group $F$ there exists an element
$\alpha \in \overline M\cap\mathfrak c$ such that $M\subset\langle \alpha \rangle$.
Let $S_0\subset F$ be the free abelian semigroup generated by the set $\mathfrak c$. The
mentioned property of the group $F$ implies that $S_0$ is countably compact. Let $g\in S_0$ be an
arbitrary non-zero element and $S$ be the closure in $S_0$ of the semigroup generated by $g$.
Since the group $F$ is countably compact, it is precompact and so topologically
periodic. Then the element $g$ is topologically periodic.
Thus $S$ is a countably compact monothetic topological monoid, which is
contained in a topological group, but which is not a group, because
$-g\in F\setminus S_0\subset F\setminus S$.

On the other hand, Tomita in~\cite{Tom2} proved that if the power $S^{2^{\mathfrak c}}$ of a
cancellative topological semigroup $S$ is countably compact
then $S$ is a group. As far as we know, a question whether there exists a
cancellative topological semigroup $S$ such that $S$ (or $S\times S$) is countably compact,
but $S$ is not a group, is still unanswered under ZFC.

By $\mathfrak u$ we denote the smallest character on a free ultrafilter on $\omega$ and by $\mathfrak p$ the smallest character of a free filter $\mathcal F$ on $\w$ that has no infinite pseudointersection (which is an infinite subset $I\subset\w$ such that $I\setminus F$ is finite for every $F\in\mathcal F$). It is known \cite[6.15]{Blass} that $\mathfrak p\le \mathrm{cf}(\mathfrak c)$ and $\omega_1\le \mathfrak p\le \mathfrak u\le\mathfrak c$. Martin's Axiom implies $\mathfrak p=\mathfrak u=\mathfrak c$ (more precisely, $\mathfrak p=\mathfrak c$ is equivalent to MA${}_{\sigma\mbox{\tiny-centered}}$, the Martin Axiom for $\sigma$-centered posets, see \cite[7.12]{Blass}).

A topological space is called {\em $\kappa$-compact} for a cardinal $\kappa$ if every its open cover of size $\le\kappa$ has a finite subcover.
In~\cite{Tom2} Tomita proved that under $\mathfrak u=\w_1$, each Tychonoff $\omega_1$-compact cancellative
topological semigroup is a group. On the other hand, he showed that
under $\mathfrak c=\mathfrak p>\omega_1$ (which is equivalent to MA${}_{\sigma\mbox{\tiny-centered}}+\neg$CH), the topological group
$\T^{\mathfrak c}$ contains an $\omega_1$-compact subsemigroup $S_0$, which is
not a group. Thus $S_0$ contains a non-periodic element $g$.
Let $S$ be the closure in $S_0$ of the semigroup generated by $g$.
Since the group $\T^{\mathfrak c}$ is compact, it is topologically periodic.
Then the element $g$ is topologically periodic.
Thus $S$ is a $\omega_1$-compact monothetic topological monoid, which is
contained in a topological group, but which is not a group, because
$-g\in \T^{\mathfrak c}\setminus S_0\subset \T^{\mathfrak c}\setminus S$.
Therefore the answer to our question
for Tychonoff $\omega_1$-compact semigroups is independent of ZFC.
\end{remark}

Theorem 4.6(i) in~\cite{Rup} implies the following
\begin{lemma}\label{lem:c-subgroup} A subgroup $G$ of a compact semitopological semigroup $S$ is
a topological group.
\end{lemma}

Remark that by Theorem 0.5 from~\cite{Rez}, a cancellative compact semitopological semigroup is a
topological group.

Pfister~\cite{Pfi} showed that each $T_3$ locally countably compact paratopological group is a
topological group. The following proposition generalizes this result a bit.

Let $\mathcal C$ be a class of spaces. A space is \emph{locally}
$\mathcal C$ if each its point has a closed neighborhood in the class $\mathcal C$.

\begin{proposition}\label{prop:cc-subgroup}
A subgroup $G$ of a locally countably compact $T_3$ topological semigroup $S$ with open
left shifts is a topological group.
\end{proposition}
\begin{proof} In order to prove that $G$ is a topological group
it suffices to show that the inversion map on $G$ is continuous at its identity $e$.
Suppose to the contrary that there exists a neighborhood $U$ of $e$ such that for
each neighborhood $V$ of $e$ there exists and element $g\in V\cap G$ with $g^{-1}\not\in U$.
Using that $S$ is a $T_3$ topological semigroup, inductively we can construct a
sequence $\{U_i\}$ of neighborhoods of $e$ such that $\overline{U_0}$ is a countable
compact subset of $U$ and $\overline{U_{i+1}^2}\subset U_i$
for each $i$. By the assumption, for each $i$ there exists a point $x_i\in U_i\cap G$
such that $x_i^{-1}\not\in U$. For each natural $k$ put $y_k=x_1\dots x_k\in U_1\dots U_k\subset
U_1^2\subset\overline{U_0}$. Let $y\in \overline{U_0}$ be a cluster point of the sequence $\{y_k\}$.
There exists $k>1$ such that $y_{k-1}\in yU_1$. Then $x_k^{-1}=y_k^{-1}y_{k-1}\in y_k^{-1}yU_1$.
Clearly, that $y^{-1}_ky$ is a cluster point of a sequence $\{y^{-1}_ky_j:j\in\mathbb N\}$.
Since, if $j>k$ then $y^{-1}_ky_j=x_{k+1}\cdots x_j\in U_{k+1}\cdots U_j\subset U_k$, we see that
$y^{-1}_ky\in\overline{U_k}\subset U_{k-1}$. So $x^{-1}_k\in y^{-1}_kyU_{1}\subset
U_{k-1}U_{1}\subset U_1U_1\subset U_0\subset U$, a contradiction with $x_k^{-1}\not\in U$.
\end{proof}

We recall that a group $G$ endowed with a topology is {\it left (resp. right) precompact},
if for each neighborhood $U$ of the identity of $G$ there exists a finite subset $F$ of $G$ such
that $FU=G$ (resp. $UF=G$). It is easy to check (see, for instance, \cite[Pr.3.1]{Rav3} or
\cite[Pr.2.1]{Rav2}) that a paratopological group $G$ is left precompact if and only if
$G$ is right precompact, so we shall call left precompact paratopological groups {\em precompact}. Moreover, it is well-known~\cite{AT} that a topological group
$G$ is precompact if and only if $G$ is a subgroup of a compact topological group. A result
from \cite{BGR} implies that a topological group $G$ is precompact if and only if
for any neighborhood $U$ of the identity of the group $G$ there exists a finite set
$F\subset G$ such that $G=FUF$.

\begin{proposition}\label{prop:lc-subgroup} A precompact subgroup $G$ of a
 locally compact semitopological semigroup $S$ is a topological group.
\end{proposition}
\begin{proof} We have that $\overline{G}$ is a locally compact semitopological semigroup.
Let $U$ be a compact neighborhood (in $\overline{G}$) of the identity of the
group $G$. Since the group $G$ is precompact, there exists a finite subset $F$ of the
group $G$ such that $G=F(U\cap G)$. Thus $FU\supset G$ and $FU$
is a closed subset of $\overline{G}$, so $FU=\overline{G}$ and the set $\overline{G}$ is compact.
Then by Lemma~\ref{lem:c-subgroup}, $G$ is a topological group.
\end{proof}

\begin{example}\label{ex:cp-ptg-not-tg}
In~\cite[Ex. 3]{BR2} it is constructed
a monothetic second countable paratopological group $G$ such that each power of $G$
is countably pracompact but $G$ is not a topological group.

Proposition~\ref{prop:lc-subgroup} does not extend to
first countable countably pracompact paratopological group $S$. Our example  will
be based on the group from~\cite[Example 3]{BR2} and the embedding described at the end
of the proof of~\cite[Theorem 3]{BR}. But for brevity and clarity we describe the
full construction here. Let $S=\mathbb T\times \mathbb T$, $\tau$ be the usual topology on $S$,
and $s:\mathbb T\to\mathbb Q$ be a (discontinuous) homomorphism onto the additive group of rational numbers.
For each natural $n$ put
$$
\begin{aligned}
V_n&=\{(g,h)\in\mathbb T\times\mathbb T:\max\{|\arg g|, |\arg h|\}<1/n\}\in\tau,\\
U_n&=\{(g,1)\in\mathbb T\times\mathbb T:0\le\arg g<1/n\}\cup \{(g,h)\in V_n:s(h)>0\}.
\end{aligned}
$$

Observing that $\{1\}\cup\{h\in\mathbb T:s(h)>0\}$ is a semigroup, it is easy to check that
the family $\{U_n\}$ satisfies Pontryagin conditions
(see,~\cite[Proposition 1.1]{Rav1}) and it is a base at the identity of a first countable
semigroup topology $\sigma$ on the group $S$.

The subgroup $G=\mathbb T\times\{1\}$ of $S$ is naturally topologically isomorphic to
Sorgenfrey circle (see Example~\ref{exam:Sorgenfrey}), so it is precompact.

We claim that the space $(S,\sigma)$ is countably pracompact.
To show this choose an element $b=e^{\psi i}\in\mathbb T$ such that $s(b)>0$. Also pick
$\varphi\in\mathbb R\setminus\psi\mathbb Q$ and put $a=e^{\varphi i}$. Put
$A=\{n(a,b):n\in\mathbb N\}\subset S$.
Example 65 from~\cite{Pon} implies that the set $A$ is dense in $(S,\tau).$
Let $U\in\sigma$ be any non-empty set. The definition of the set $U_n$
implies that there exists a non-empty set $V\in\tau$ and a number $k\in N$
such that $\{(g,h)\in V: s(h)>k\}\subset U$. The density of
$A$ in $(D,\tau)$ implies that there exists $m\in N$ such that
$ms(b)>k$ and $m(a,b)\in V$. Then $m(a,b)\in U$. Thus the set $A$ is dense in $(S,\sigma)$.

Now let $B$ be any infinite subset of $A$. Let $(g,h)$ be an accumulation point of $B$
is $(S,\tau)$. Let $(g,h)+U_n\in\sigma$ be a basic neighborhood of $(g,h)$.
A set $B_n=((g,h)+V_n)\cap B$ is infinite. Since $s(b)>0$, the set
$C_n=\{(c,d)\in B_n: s(d)\le s(h)\}$ is finite. Since $B_n\setminus C_n\subset (g,h)+U_n$,
we conclude that $(g,h)$ is an accumulation point of $B$ is $(S,\sigma)$.
\end{example}

There exists a pseudocompact quasitopological group $G$ of period $2$,
which is not a paratopological group, see (\cite{Kor1},~\cite{Kor2}).
On the other hand, Reznichenko in~\cite[Theorem 2.5]{Rez} showed
that each semitopological group $G\in \mathcal N$ is a
topological group, where $\mathcal N$ is a family of all pseudocompact spaces $X$ such that
$(X,X)$ is a \emph{Grothendieck pair}, that is
if each continuous image of $X$ in $C_p(Y)$ has the compact closure in $C_p(Y)$.
In particular, a pseudocompact space $X$ belongs to $\mathcal N$ provided
$X$ has one of the following properties: countable compactness, countable tightness,
separability, $X$ is a $k$-space, see~\cite{Rez}.

We can prove a straightforward generalization of Pontryagin's theorem for compact-like groups.

\begin{proposition}\label{prop:local-to-global} Let $\mathcal C$ be a class of feebly compact
spaces closed with respect to
homeomorphisms. Moreover, each space $X$ belongs to $\mathcal C$ provided there exists finitely
many its subspaces $X_1,\dots, X_n\in\mathcal C$ such that for any point $x\in X$ there exists
$X_i$ which is a neighborhood of $x$. Then each locally $\mathcal C$ monothetic
topological $G$ group is $\mathcal C$.
\end{proposition}
\begin{proof} Theorem 6.9.27 from~\cite{AT} implies that Ra\v\i kov completion
$\hat G$ of $G$
is a locally compact topological group. Since $\hat G$ is monothetic, it is compact by
Pontryagin's theorem. The group $G$ is precompact as a subgroup of a precompact topological group by
Proposition 3.7.4 from~\cite{AT}. Since the group $G$ is locally $\mathcal C$,
there exists an open neighborhood $U$ of the identity $e$ of $G$ such that $\overline{U}\in\mathcal
C$. Since the group $G$ is precompact, there exists a finite subset $F$ of $G$ such that
$G=FU$. Now the property of the class $\mathcal C$ implies $G\in\mathcal C$.
\end{proof}


\begin{thebibliography}{}

\bibitem[AT]{AT}
A. Arhangel'skii, M. Tkachenko,
{\it Topological groups and related structures},
Atlantis Studies in Mathematics, Atlantis Press, Paris;
World Scientific Publishing Co. Pte. Ltd., Hackensack, NJ, 2008.

\bibitem[Ave1]{Ave1}
B. Averbukh,
{\it On unitary Cauchy filters on topological monoids},
Top. Algebra Appl. {\bf 1} (2013), 46--59.

\bibitem[Ave2]{Ave2}
B. Averbukh,
{\it On finest unitary extensions of topological monoids},
Top. Algebra Appl. {\bf 3} (2015), 1--10.

\bibitem[Ave3]{Ave3}
B. Averbukh,
{\it On unitary extensions and unitary completions of topological monoids},
Top. Algebra Appl. {\bf 4} (2016), 18--30.

\bibitem[Ave4]{Ave4}
B. Averbukh,
{\it A criterion of the existence of an embedding of a monothetic monoid into a topological group}, Top. Algebra Appl. (to appear).

\bibitem[Ban]{Ban}
T. Banakh,
{\it A homogeneous first-countable zero-dimensional compactum failing to be a left-topological group},
Matematychni Studii, {\bf 29} (2008) 215--217,
\url{http://matstud.org.ua/texts/2008/29_2/29_2_215_217.pdf}

\bibitem[BDG]{BDG}
T. Banakh, S. Dimitrova, O. Gutik,
{\it Embedding the bicyclic semigroup into countably compact topological semigroups},
Topology Appl. {\bf 157}:18 (2010) 2803--2814,
\url{http://arxiv.org/abs/0811.4276}

\bibitem[BGR]{BGR}
T. Banakh, I. Guran, A. Ravsky,
{\it Boundedness and separability in topological groups},
(in preparation).

\bibitem[BR]{BR}
T. Banakh, A. Ravsky,
{\it On subgroups of saturated or totally bounded paratopological groups},
Algebra and Discrete Mathematics {\bf 4} (2003), 1-20.
\url{http://arxiv.org/abs/1003.5355}

\bibitem[BR2]{BR2}
T. Banakh, A. Ravsky,
{\it Feebly compact paratopological groups},
\url{https://arxiv.org/abs/1003.5343v5}

\bibitem[BGR2]{BGR2}
S. Bardyla, O. Gutik, A. Ravsky,
{\it H-closed quasitopological groups}
Topology Appl. {\bf 217} (2017), 51--58.
\url{http://arxiv.org/abs/1506.08320}

\bibitem[Bla]{Blass} A.~Blass,
{\it Combinatorial cardinal characteristics of the continuum},
Handbook of set theory. Vols. 1, 2, 3, 395--489, Springer, Dordrecht, 2010.

\bibitem[BG]{BG}
B. Bokalo, I. Guran,
{\it Sequentially compact Hausdorff cancellative semigroup is a topological group},
Matematychni Studii {\bf 6} (1996), 39--40.
\url{http://matstud.org.ua/texts/1996/6/6_039-040.pdf}

\bibitem[CHK]{CHK}
J.H. Carruth, J.A. Hildebrant, R.J. Koch,
{\it The theory of topological semigroups},
Pure and Applied Mathematics, Marcel Dekker, Inc., 1983, vi + 244 pp.


\bibitem[CP]{CP}
A.H.~Clifford, G.B.~Preston,
{\it The Algebraic Theory of Semigroups}, Vol. I
Providence, R.I.,  1964 (Russian translation, Moskow, Mir, 1972).

\bibitem[DS]{DS}
D. Dikranjan, D. Shakhmatov,
{\it Selected topics in the structure theory of topological groups},
in: Open Problems in  Topology, II (E. Pearl ed.), Elsevier, 2007, p.389--406.

\bibitem[D-AS]{Dorantes-AldamaShakhmatov}
A. Dorantes-Aldama, D. Shakhmatov,
{\it Selective  sequential  pseudocompactness},
Top. Appl. \textbf{222} (2017), 53--69.


\bibitem[vDRRT]{vanDouwenReedRoscoeTree1991}
E. van Douwen, G. Reed, A. Roscoe, I.Tree,
{\it Star covering properties}
Top. Appl. \textbf{39}:1 (1991), 71--103.

\bibitem[Ell]{Ell}
R. Ellis,
{\it  Continuty and homeomorphism groups},
Proc. Amer. Math. Soc. {\bf 4} (1953), 969--973.

\bibitem[Ell2]{Ell2}
R. Ellis,
{\it A note on the continuity of the inverse},
Proc. Amer. Math. Soc. {\bf 8} (1957), 372--373.

\bibitem[Eng]{Eng}
R. Engelking,
{\it General Topology},
Heldermann Verlag, Berlin, 1989.

\bibitem[GK]{GK}
I. Guran, M. Kisil',
{\it Pontryagin's alternative for locally compact cancellative monoids},
Visnyk of the Lviv Univ. Series Mech. Math. No.77 (2012), 84--88, (in Ukrainian),
\url{http://prima.lnu.edu.ua/faculty/mechmat/Departments/MathVisnykLU/VLUsMath-77/VisnM-77-084.pdf}

O. Gutik, A. Ravsky,
{\it On feebly compact inverse primitive (semi)topological semigroups},
Matematychni Studii. {\bf 44}:1 (2015), 3-26,
\url{http://matstud.org.ua/texts/2015/44_1/3-26.html}

\bibitem[GR2]{GR2}
O. Gutik, A. Ravsky,
{\it On old and new classes of feebly compact spaces},
Visnyk Lviv Univ. Ser. Mech. Mat. {\bf 85} (20189), (to appear)

\bibitem[Hew]{Hew}
E. Hewitt,
{\it Compact monothetic semigroups},
Duke Math. J., {\bf 23} (1956), 447--457,
\url{https://projecteuclid.org/euclid.dmj/1077466957}

P. Kenderov, I. Kortezov, W. Moors,
{\it Topological games and topological groups}
Topology Appl. {\bf 109}:2 (2001), 157--165,
\url{https://www.sciencedirect.com/science/article/pii/S0166864199001522}

\bibitem[Koc]{Koc}
R. Koch,
{\it On monothetic semigroups},
Proc. Amer. Math. Soc., {\bf 8} (1957), 397--401.

\bibitem[Kor1]{Kor1}
A. Korovin,
{\it Continuous actions of Abelian groups and topological properties in $C_p$-theory},
Ph.D. Thesis, Moskow State University, Moskow (1990) (in Russian).

\bibitem[Kor2]{Kor2}
A. Korovin,
{\it Continuous actions of pseudocompact groups and the topological group axioms},
Deposited in VINITI, \#3734-D, Moskow (1990) (in Russian).

\bibitem[KTW]{KTW}
P. Koszmider, A. Tomita, S. Watson,
{\it Forcing countably compact group topologies on a larger free Abelian group},
Topology Proc. {\bf 25} (2000), 563--574.


\bibitem[Kur]{Kur}
A. Kurosh,
{\it Group theory},
M.:Nauka, 1967, (in Russian).

\bibitem[LLZ]{LLZ}
K. Lau, J. Lawson, W. Zeng,
{\it Embedding Locally Compact Semigroups into Groups},
Semigroup Forum. {\bf 56} (1998), 151--156.

\bibitem[Law]{Law}
J. Lawson,
{\it Embedding semigroups into Lie groups},
The Analytical and Topological Theory of Semigroups: Trends and Developments,
editors K. H. Hofmann, J. D. Lawson, J. S. Pym. Walter de Gruyter, 1990.

\bibitem[Lip]{Lipparini2016}
P. Lipparini,
{\it The equivalence of two definitions of sequential pseudocompactness},
Appl. Gen. Topol. \textbf{17}:1 (2016), 1--5.
(long version \url{http://arxiv.org/abs/1201.4832})

\bibitem[MaT]{MT}
R. Madariaga-Garcia, A. Tomita,
{\it Countably compact topological group topologies on free Abelian groups from selective ultrafilters},
Topology Appl. {\bf 154} (2007), 1470--1480.

\bibitem[Mat]{Matveev1998} M.~Matveev,
{\it A Survey on Star Covering Properties},
preprint; (available at \url{http://at.yorku.ca/v/a/a/a/19.htm}).




\bibitem[McKil]{McKil}
S. McKilligan,
{\it Embedding topological semigroups in topological groups},
Proc. Edinburgh Math. Soc. {\bf 17} (1970/71), 127--138.


\bibitem[Pfi]{Pfi}
H. Pfister,
{\it Continuity of the inverse},
Proc. Amer. Math. Soc. {\bf 95} (1985), 312--314.

\bibitem[Pon]{Pon}
L. Pontrjagin,
{\it Continuous groups}, 2nd ed.,
M.: GITTL, (1954), (in Russian).

\bibitem[Pon2]{Pon2}
L. Pontrjagin,
{\it Continuous groups},
M.:Nauka, 1973, (in Russian).

\bibitem[Rav1]{Rav1}
A. Ravsky,
{\it Paratopological groups I},
Matematychni Studii. {\bf 16} (2001), 37--48.

\bibitem[Rav2]{Rav2}
A. Ravsky,
{\it Paratopological groups II},
Matematychni Studii. {\bf 17} (2002), 93--101.

\bibitem[Rav3]{Rav3}
A. Ravsky,
{\it The topological and algebraical properties of paratopological groups},
Ph.D. Thesis. Lviv University, (2002) (in Ukrainian).

\bibitem[Rez]{Rez}
E. Reznichenko,
{\it Extension of functions defined on products of pseudocompact spaces and continuity of the
inverse in pseudocompact groups},
Topology Appl. {\bf 59} (1994), 233--244,
\url{https://www.sciencedirect.com/science/article/pii/0166864194900213}

\bibitem[Rot]{Rot}
N. Rothman,
{\it Embedding of topological semigroups},
Math. Ann. {\bf 139} (1960), 197--203.

\bibitem[Rup]{Rup}
W. Ruppert,
{\it Compact Semitopological Semigroups: An Intrinsic Theory},
Lecture Notes in Math. {\bf 1079}, Springer-Verlag, 1984.

\bibitem[Ste]{StephensonJr1984}
R.~Stephenson, Jr,
{\it Initially $\kappa$-compact and related compact spaces}, in K. Kunen, J. E. Vaughan (eds.), Handbook of Set-Theoretic Topology, Elsevier, 1984,
P. 603--632.


\bibitem[Tka]{Tka}
M. Tkachenko,
{\it Countably compact and pseudocompact topologies on free Abelian groups},
Soviet Math. (Iz. VUZ) {\bf 34}:5 (1990), 79--86.


\bibitem[Tom]{Tom2}
A. Tomita,
{\it The Wallace Problem: a counterexample from $MA_{countable}$ and $p$-compactness},
Canadian Math. Bulletin, {\bf 39}:4 (1996), 486--498.

\bibitem[Vau]{VaughanHSTT}
J.~Vaughan,
{\it Countably compact and sequentially compact spaces}, in K. Kunen, J. E. Vaughan (eds.), Handbook of Set-Theoretic Topology, Elsevier, 1984,
P. 569--602.


\bibitem[Zel]{Zel}
E. Zelenyuk,
{\it To Pontrjagin alternative for topological semigroups},
Mat. zametki {\bf 44}:3 (1988), 402--403, (in Russian).

\bibitem[Zel2]{Zel2}
Y. Zelenyuk,
{\it A locally compact noncompact monothetic semigroup with identity},
Fund. Math., published online (DOI:10.4064/fm535-3-2018).


\end{thebibliography}
\end{document}